\documentclass[12pt,a4paper]{amsart}

\usepackage{amssymb, amstext, amscd, amsmath, color}

\usepackage{url}
\usepackage{tikz}
\usepackage{epstopdf}
\usepackage{amsfonts}
\usepackage{amsthm}
\usepackage{amsfonts}
\usepackage{array}
\usepackage{epsfig}
\usepackage{eucal}
\usepackage{latexsym}
\usepackage{mathrsfs}
\usepackage{textcomp}
\usepackage{verbatim}
\usepackage{setspace}

\newtheorem{thm}{Theorem}[section]

\newtheorem{cor}[thm]{Corollary}

\newtheorem{defn}[thm]{Definition}

\newtheorem{lem}[thm]{Lemma}

\newtheorem{rem}[thm]{Remark}

\numberwithin{equation}{section}

\newcommand{\bD}{{\mathbb{D}}}

\newcommand{\bR}{{\mathbb{R}}}

  \newcommand{\C}{{\mathcal{C}}}

  \newcommand{\F}{{\mathcal{F}}}

  \newcommand{\M}{{\mathcal{M}}}

\renewcommand{\P}{{\mathcal{P}}}

\renewcommand{\S}{{\mathcal{S}}}
  \newcommand{\T}{{\mathcal{T}}}
  \newcommand{\U}{{\mathcal{U}}}

  \newcommand{\X}{{\mathcal{X}}}
  \newcommand{\Y}{{\mathcal{Y}}}
  
\newcommand{\rank}{\operatorname{rank}}
\usepackage{blkarray}

\begin{document}

\title[Double-distance frameworks and mixed sparsity graphs]{Double-distance frameworks and mixed sparsity graphs}
\author[A. Nixon]{A. Nixon}
\address{Dept.\ Math.\ Stats.\\ Lancaster University\\
Lancaster LA1 4YF \\U.K.}
\email{a.nixon@lancaster.ac.uk}
\author[S.C. Power]{S. C. Power}
\address{Dept.\ Math.\ Stats.\\ Lancaster University\\
Lancaster LA1 4YF \\U.K. }
\email{s.power@lancaster.ac.uk}
\thanks{
This work was supported by the Engineering and Physical Sciences Research Council [grant number\quad  EP/P01108X/1].}

\begin{abstract}
A rigidity theory is developed for frameworks in a metric space with two types of distance constraints. Mixed sparsity graph characterisations are obtained for the infinitesimal and continuous rigidity of completely regular bar-joint frameworks in
a variety of such contexts. The main results are combinatorial characterisations for
(i) frameworks restricted to surfaces with both Euclidean and geodesic distance constraints, (ii) frameworks in the plane with Euclidean and non-Euclidean distance constraints, and  (iii) direction-length frameworks in the non-Euclidean plane.
\end{abstract}

\date{}
\maketitle

\section{Introduction}
A bar-joint framework $(G,p)$ in $\mathbb{R}^d$ is the combination of a finite simple graph $G=(V,E)$ and a map $p:V\rightarrow \mathbb{R}^d$. The framework is rigid if the only edge-length-preserving continuous deformations of the framework arise from isometries of $\mathbb{R}^d$.
While it is typically hard to determine rigidity for a given framework, characterising generic rigidity and infinitesimal rigidity is a well-studied problem, solved in the Euclidean plane by   Pollaczek-Geiringer \cite{pol} and Laman  \cite{laman}.
We consider here combinatorial characterisations of rigidity for bar-joint frameworks for various multi-distance  and multi-constraint settings.
When there are just two types of distance constraint
the underlying structure graph of the framework may be bi-coloured with blue and red edges, representing  
the two types of constraint. In this case rigidity requirements lead naturally to necessary sparsity conditions for  monochrome subgraphs, which we refer to as mixed sparsity conditions.

We obtain combinatorial characterisations of generic rigidity for  double-distance frameworks in various 2-dimensional contexts. 
These include (i) bar-joint frameworks restricted to surfaces with both Euclidean and geodesic distance constraints, (ii) frameworks in the plane with Euclidean and non-Euclidean distance constraints, and (iii) direction-length frameworks in the non-Euclidean plane. 
%and (iv)frameworks on a vertical cone with additional horizontal distance constraints.
The proofs  make use of inductive characterisations of bi-coloured structure graphs, of the appropriate mixed sparsity type, together with determinations of minimal
rigidity preservation  for a range of coloured graph Henneberg extension moves. 
These new contexts in geometric rigidity theory and their combinatorial characterisations extend {the analyses in: Nixon, Owen and Power \cite{nix-owe-pow-1, nix-owe-pow-2} of $3$-dimensional frameworks which are vertex-constrained to surfaces; Kitson and Power \cite{kit-pow-1} of $2$-dimensional frameworks with non-Euclidean distances; and Servatius and Whiteley \cite{SW} of $2$-dimensional (Euclidean) direction-length frameworks. }

Frameworks in $\bR^2$ with both distance and direction constraints \cite{SW} fall into the more general category of multiple-constraint frameworks. This is also the case for a range of more elementary frameworks in product type contexts where  the individual constraints depend on independent variables and so, in this sense, are separable.
  It seems to us, moreover, that multiple constraint rigidity theory  is of potential significance in a range of applications. We note, for example, that in three dimensions, the measure of residual dipolar coupling (RDC) between rigid units of a protein may be interpreted as a secondary nonmetric constraint \cite{KSC}. Also, in the area of 3D sensor networks it is natural to consider the augmentation of Euclidean distances by (possibly partially available) data such as altitudes to a reference surface \cite{lib-lav}. An elementary example of this, considered in Section \ref{ss:furthercontexts}, is the
separable double-distance context $(\bR^3, d_{xy}, d_z)$ associated with projected distances in the $xy$-plane and in the $z$-axis.  

 The discussion is organised as follows. In Section \ref{s:doubledistance} we define double-distance contexts and the minimal infinitesimal rigidity of their bar-joint frameworks, and we give in Theorem \ref{t:topcount} the necessary Maxwell counting condition for the edges and vertices of such frameworks. In Section \ref{ss:furthercontexts} we recall the $(2,k)$-sparsity and $(2,k)$-tight conditions for graphs and we give 3 illustrative double-distance contexts and their associated sparsity requirements.  In Sections \ref{s:cylinder},   \ref{s:EuclideanAndNon} and \ref{s:directionlength}, %\ref{s:cone-with-projection} 
we obtain the main results, namely the characterisation of infinitesimal rigidity 
for (completely regular) frameworks on some surfaces (with direct and geodesic distances), for frameworks in the plane (with Euclidean and non-Euclidean distances), and for direction-length frameworks in the non-Euclidean plane.
%and, finally, for frameworks with joints restricted to a vertical cone with additional horizontal constraints.  
In each of these sections we provide a recursive construction of the appropriate class of bi-coloured graphs. We then give a geometric analysis that the relevant graph operations preserve minimal rigidity at the level of completely regular frameworks.

In the final section we indicate six additional multi-distance contexts for further analysis and which generally require a deeper combinatorial and geometric analysis. This includes Euclidean distance frameworks which are augmented with constraints associated with projections and reflections. 

\section{Double-distance bar-joint frameworks}\label{s:doubledistance}

A  \emph{separation constraint}, or \emph{separation distance},  associated with a set $X$ of points $x,y,\dots $, is a continuous non-negative  function $d:X \times X \to \bR_+$ which is symmetric, so that $d(x,y)=d(y,x)$, for all $x,y$. 
In particular we do not require the triangle inequality to hold and $d(x,x)$ may be nonzero. {If $d_1, d_2$ are separation distances for $X$ then 
$(X, d_1, d_2)$ is said to be a \emph{double-constraint context} or \emph{double-distance context.}

\subsection{Continuous rigidity} Bar-joint frameworks $(G,p)$ and their continuous flexes  may be defined for any metric space $(X,d)$.  In this case $G$ is a simple graph and $p:V\to X$ is a placement map, or realisation in $X$. It is assumed that $p(v)$ and $p(w)$ are distinct if $vw$ is an edge. A continuous flex $p(t)$, for a time parameter $t$ in some interval $[0,a]$, is a pointwise continuous path of placements $p_t:V \to X$, with $p_0=p$, such that for each edge $e=vw$ the map $t \to d(p_t(v),p_t(w))$ is constant.   There are generally two variants of the notion  of a \emph{trivial continuous flex}. 
The (possibly weaker) first form requires that the flex extends to a flex of the complete framework associated with $(G,p)$. 
One could say that the framework is \emph{distance rigid} or \emph{separation rigid} if each continuous flex is trivial in this sense over a small enough time interval.
The other (possibly stronger) form of continuous rigidity requires that any continuous flex over a small enough time interval is given by a continuous \emph{isometric motion} $\tilde{p}_t:X \to X$ of the entire metric space.
We say that the framework is \emph{spatially rigid} in this case, or, simply, \emph{continuously rigid}. 

Bar-joint frameworks restricted to surfaces (see \cite{nix-owe-pow-1,nix-owe-pow-2}) as well as bar-joint frameworks in non-Euclidean spaces (see \cite{kit,kit-pow-1}) give examples of bar-joint frameworks in a metric space setting. 

One can readily extend the definitions of continuous flexes and  continuous rigidity above to \emph{double-distance frameworks} {in a double-distance context $(X, d_1, d_2)$ where $X$ is a topological space with topology determined by $d_1, d_2$.}
Such a framework, $(G,p)$, is associated with  a placement map $p:V\to X$,  and the structure graph $G=(V,E)$ is a 2-(edge)-coloured (or \emph{bi-coloured}) multigraph whose monochrome subgraphs are simple graphs, possibly with loops. We refer to these graphs as  monochrome simple bi-coloured graphs. 

Formally, a \emph{bar} of $(G,p)$ associated with a coloured edge $e\in E$ of $(G,p)$ is a  triple $p(e)=\{p(v), p(w); c\}$, where $e = (vw,c)$ is an edge $vw$ of  $G$ with colour $c= c(e)$. The \emph{length} or \emph{separation} of the bar $p(e)$ is then the nonnegative real number $d(p(e))= d_c(p(v), p(w))$, where $d_c$ is the distance function for the colour $c$. We shall use the colours blue and red, indicated also by $b$ and $r$.

The definition of a continuous flex is as before but with the requirement for a trivial continuous flex being agreement with a continuous motion $\tilde{p}_t$ of $X$ preserving both distance functions. %$d_1$ and $d_2$. 
As usual a continuous flex may be considered as a continuous path in the topological space $V(G,p)$ of all \emph{configurations} which, by definition, is the set of all solutions $q: V \to X$ to the set of constraint equations
\[
d(q(e)) = d(p(e)),\quad e\in E(G).
\]

While these generalities hold for arbitrary topological spaces we shall be concerned only with Euclidean spaces and associated finite-dimensional manifolds.

\subsection{Infinitesimal flexibility and rigidity}
We now extend the notion of infinitesimal rigidity for bar-joint frameworks. The following abstract class of double-distance contexts covers all of our main examples. 

A double-distance context   $(X, d_1, d_2)$ is \emph{essentially smooth} if
\medskip

(i) $X$ is a topological space with a  dense open subset $X_0$, possibly equal to $X$, which is a compatible smooth manifold ($C_\infty$-manifold),

(ii) the separation distances $d_1, d_2$ are differentiable functions on $X_0 \times X_0$ which determine the topology of  $X_0$.
\medskip

If $X_0=X$ then we say that  $(X, d_1, d_2)$ is a \emph{smooth} double-distance context. This is the case for  the double-distance frameworks on the surfaces considered in Section \ref{s:cylinder}, and for the
mixed-norm frameworks in $\bR^2$ in Section \ref{s:EuclideanAndNon}.

\begin{defn}\label{d:infflex}Let $\X= (X, X_0, d_1, d_2)$ be an essentially smooth double-distance context where $X_0$ is an open subset of the Euclidean space $\bR^n$ and $d_1(x,y)= \|x-y\|_2$.

(i) An \emph{infinitesimal flex} of the bar $\{p_1,p_2;d_i\}$, where $p_1, p_2 \in X_0$, is a \emph{velocity vector} $u = (u_1, u_2)$ in $\bR^2 \oplus \bR^2$ such that
\[
d_i(p_1+tu_1, p_2+tu_2)-d_i(p_1,p_2)) = o(t), \quad \mbox{as } t \to 0.
\]
 
(ii) An  \emph{infinitesimal 
flex} of a double-distance bar-joint framework $(G,p)$ for 
$\X$ with joints in $X_0$ is a velocity vector $u : V \to \bR^{2|V|}$ such that the restriction of
$u$ to each bar of $(G,p)$ is an infinitesimal flex.
\end{defn}

The infinitesimal flex condition may be paraphrased as the statement that  \emph{to first order the velocity vectors do not distort bar separations}. In the case that $d_1(x,y)$ is a Euclidean metric note that a velocity pair is an infinitesimal flex for 
the bar $\{p_1, p_2; d_1\}$ if and only if it is an infinitesimal flex for the length of $\{p_1, p_2; d_1^2\}$. This  is a convenient polynomial separation distance and the equivalence follows from the identity $a^2-b^2=(a-b)(a+b)$.

Let $(G,p)$ be a bar-joint framework for the
double-distance context $(X, X_0, d_1, d_2)$ with edge set $E=E_1\cup E_2$ and placement $p$ with range in $X_0^{|V|}$.
Define the \emph{rigidity map} for the structure graph $G$ to be the function
$f_G: \prod_{v\in V} X_0 \to \prod_{e\in E} \bR$ given by
\[
 f_G((x_v)_{v\in V}) =  \prod_{e \in E} d(p(e)),
\]
where $d(p(e))$ is the separation $d_i(x_v, x_w)$ for the bar $p(e)=\{x_v, x_w; d_i\}$.
 
The following lemma identifies  the space of infinitesimal flexes with the kernel of the derivative of the rigidity map.
The proof is the same as the usual proof for bar-joint frameworks in $\bR^n$.

\begin{lem}\label{l:Df_lemma} Let $(X, X_0, d_1, d_2)$ be an essentially smooth double-distance context where $X_0$ is an open subset of the Euclidean space $\bR^n$ and $d_1(x,y)= \|x-y\|_2$.
Then a vector $u$ is an infinitesimal flex for $(G,p)$ if and only if
the matrix product $((Df_G)(p))u^t$, for the transpose vector $u^t$, is the zero vector in
$\bR^{|E|}$.
\end{lem}

The lemma shows that we could have defined the real vector space of infinitesimal flexes as the kernel of a rigidity matrix $R(G,p)=(Df_G)(p)$ determined by $d_1$ (or $d_1^2$) and $d_2$. However both viewpoints are useful since the rigidity map may only be determined implicitly as a differentiable function.

Ordering the rows of $Df_G$ so that the rows labelled by edges of the subset $E_1$ for the distance constraint $d_1$ come first,  we have

\[
Df_G(p) =
\begin{blockarray}{cc}
\begin{block}{c[c]}
E_1 & 2R_1(p) \\
E_2 & R_2(p) \\
\end{block}
\end{blockarray}
 \]
where $R_1(p)$ is the usual Euclidean distance rigidity matrix
for the framework $(G_1,p)$, where $G_1=(V, E_1)$,
and where $R_2(p)$ may be considered as the rigidity matrix for the $d_1$-separation distance framework $(G_2,p)$, where $G_2 =(V, E_2)$.

We give an explicit example of the form of a rigidity matrix in Example 2 of Section \ref{ss:furthercontexts}.

We now define the notion of infinitesimal rigidity for a bar-joint framework $(G,p)$ in an essentially smooth double-distance context. 

\begin{defn}\label{d:infinitesimalflex} Let $\X= (X, X_0, d_1, d_2)$ be an essentially smooth double-distance context and let $(G,p)$ be a double-distance bar-joint framework for $\X$ with joints in $X_0$. 

(i) An infinitesimal flex $u$ of  $(G,p)$  is a \emph{rigid motion infinitesimal flex} if
and only if $u$ is the restriction of an infinitesimal flex of the framework $(K(V),p)$ where $K(V)$ is the complete bi-coloured graph on $V$.

(ii) $(G,p)$ is infinitesimally rigid if every infinitesimal flex is a rigid motion infinitesimal flex.
\end{defn}

In analogy with our comments on separation rigidity this definition could be referred to as \emph{separation infinitesimal rigidity}. Here $K(V)$ is the graph with 2 edges between each pair of vertices and 2 loop edges on each vertex, these pairs having distinct colours. The stricter form of rigid motion infinitesimal flex (spatially rigid) would require that  $u$
 is the restriction of an infinitesimal flex of $K(X_0)$.

\begin{defn}\label{d:regular}
Let $\X= (X, X_0, d_1, d_2)$ be an essentially smooth double-distance context. A double-distance framework $(G,p)$ with joints in $X_0$  is \emph{regular} if the rank of $Df_G(p)$ is maximal over all  frameworks $(G,q)$ with joints in $X_0$.
\end{defn}

It follows trivially from the definition of infinitesimal rigidity that the complete graph framework $(K(V),p)$ is infinitesimally rigid. In all our contexts we see that the regular frameworks 
$(K(V),p)$ are also continuously rigid, except for certain small graphs, and so we may make use of the following equivalence. 
We omit the proof which is a standard implicit function argument  \cite{asi-rot}. See also Theorem 3.8 of \cite{nix-owe-pow-1} on the equivalence of continuous rigidity and infinitesimal rigidity for regular frameworks constrained to algebraic surfaces.

\begin{thm}\label{t:rigidityequivalence}
Let  $\X= (X, X_0, d_1, d_2)$ be an essentially smooth double-distance context, with $(X_0, d_1)$  a metric space,  let  $(G,p)$ be a regular framework
with multigraph $G=(V,E)$ and suppose that $(K(V),p)$ is continuously rigid. Then $(G,p)$   is infinitesimally rigid if and only if it is continuously rigid.
\end{thm}

A bar-joint framework $(G,p)$ is \emph{minimally infinitesimally rigid} if it is infinitesimally rigid and $(G-e,p)$ has a nontrivial infinitesimal flex for every edge $e$. 
The following theorem generalises the classical Maxwell counting condition, or ``top count'', which is necessary for minimal infinitesimal rigidity. 

\begin{thm}\label{t:topcount}
Let  $\X= (X, X_0, d_1, d_2)$ be an essentially smooth double-distance context, where $X_0$ is a connected manifold and $(X_0, d_1)$ is  a metric space. Let
$(G,p)$ be a 
minimally infinitesimally rigid framework for $\X$, with joints in $X_0$, and let $G=(V,E)$.
Then 
\[
|E|= \dim(X_0)|V|- \dim(\F((K(V),p)))
\]
where $\F((K(V),p))$ is the infinitesimal flex space for the framework for the complete multigraph $K(V)$.
\end{thm}

\begin{proof}
The rigidity matrix $D_{f_G}(p)$ has size $|E| \times n$ where $n=\dim(X_0)|V|$. By Lemma \ref{l:Df_lemma} and Definition \ref{d:infinitesimalflex} the column rank of the rigidity matrix is equal to  $n- \dim(\F((K(V),p)))$. By the minimality condition the rows are linearly independent and so the row rank is $|E|$.
\end{proof}

The theorem expresses a maximal rank property of minimal infinitesimal rigidity. The following definition of \emph{complete regularity} in terms of ranks for subframeworks may be viewed as a natural strong form of geometric genericness. This will be a convenient assumption for the frameworks in construction moves $(G,p)\to (G',p')$ where we wish to show the preservation  of infinitesimal rigidity.

\begin{defn}
Let $(G,p)$ be a framework for a double-distance context and let $K$ be the complete multigraph on the vertex set of $G$. Then $(G,p)$ is completely regular if every subframework $(H, p|H)$ of $(K,p)$ has a rigidity matrix with maximal rank over all framework realisations of $H$.
\end{defn}

If the rigidity map for an essentially smooth double-distance context is a real analytic function  on a connected  dense open set   then it follows from standard arguments that the set of completely regular frameworks for a structure graph $G$ is a dense open set amongst the set of all frameworks for $G$. 

\subsection{Examples of contexts and sparsity conditions} \label{ss:furthercontexts}
For $k\in \{1,2,3\}$, a multigraph $G=(V,E)$ is \emph{$(2,k)$-sparse} if $|E'|\leq 2|V'|-k$ for all subgraphs $(V',E')$ with $|E'|>0$.
%(when $k=3$ we only ask the inequality to hold when $|E'|>0$). 
Moreover $G$ is \emph{$(2,k)$-tight} if $|E|=2|V|-k$ and $G$ is $(2,k)$-sparse. 

Note that: (i) a $(2,3)$-sparse multigraph in fact has no parallel edges or loops, that is, it is a simple graph. (ii) a $(2,2)$-sparse graph is loopless and may have parallel edges of multiplicity $2$ only. (iii) a $(2,1)$-sparse multigraph may have loops, but no multiple loops at a single vertex, and may have parallel edges of multiplicity $2$ or $3$ only.

For bi-coloured multigraphs arising from minimally rigid bar-joint frameworks in double-distance contexts with $\dim X_0 =2$ we shall see that there are further conditions on monochrome subgraphs.

We now give some illustrative examples of double-distance contexts and multi-constraint contexts.
\medskip

1. It should be evident that the discussion above extends in a routine way to multi-constraint contexts and their bar-joint frameworks, and that  the rigidity matrix then takes the form of a column of rigidity matrices $R_1, \dots , R_n$ associated with separation distances $d_1, \dots , d_n$ respectively. Define a \emph{separable} multi-distance context to be a multi-distance context in which the functions $d_i$ have independent variables. 

Consider in particular  the  multi-seminorm context $(\bR^n,d_1^{r_1},d_2^{r_2}, \dots, d_k^{r_k})$, associated with the decomposition $\bR^n= \bR^{r_1} \times
\dots \times \bR^{r_k}$ for a partition of $\{1, \dots ,n\}$ into sets
$S_1, \dots , S_k$, where
$d_i^{r_i}(x,y)$ depends only on the variables for the set $S_i$.
%= \sum_{j \in S_i}|x_j - y_j|$.
Then the following elementary theorem follows from the block diagonality of the rigidity matrix.

\begin{thm}
A completely regular separable multi-seminorm framework in $\bR^n$  is minimally infinitesimally rigid if and only if each 
%$i^{r_i}$-
{ maximal } monochrome subframework is minimally infinitesimally rigid in $(\bR^{r_i}, d_i^{r_i})$.
\end{thm}

Using the well-known combinatorial characterisations of infinitesimal rigidity in 1- and 2-dimensions \cite{laman} we obtain the following corollary, where each semi-norm $d_i^{r_i}$ corresponds to a Euclidean norm for $1$ or $2$ of the variables. 
%$d_i^{r_i}(x,y) =\sum_{j \in S_i}|x_j - y_j|$. 

\begin{cor}
Let $(\bR^n,d_1^{r_1},d_2^{r_2}, \dots, d_k^{r_k})$ be a separable multi-seminorm context in which $r_i\leq 2$ for $1\leq i \leq k$. Then a generic separable multi-seminorm framework in $\bR^n$  is minimally infinitesimally rigid if and only if each 
$d_i^1$-
{maximal} monochrome structure graph is a spanning tree and each $d_i^2$-monochrome structure graph is $(2,3)$-tight.
\end{cor}

We note two natural specific cases. Firstly consider the multi-seminorm context $(\bR^n,d_1,\break\dots , d_n)$ where  $d_i(x,y)=|x_i-y_i|$ for $x=(x_1,\dots , x_n), y=(y_1,\dots , y_n)$. Here minimal infinitesimal rigidity is characterised by the $n$-coloured multigraph being the union of $n$ monochrome spanning trees of distinct colours. Secondly, consider the  double-distance context
$(\bR^3, d_{xy}, d_z)$ indicated in the introduction. In this case one
can deduce readily that an algebraically generic bar-joint framework $(G,p)$ is minimally infinitesimally rigid if and only if its $xy$-projection is minimally infinitesimally rigid in $\bR^2$ and its $z$-axis projection is minimally infinitesimally rigid in $\bR^1$. 
\medskip

2. Recall that direction-length frameworks in the plane are frameworks with both distance constraints and direction constraints between pairs
 of joints, see Servatius and Whiteley \cite{SW}. Thus they  are associated with realisations of simple bi-coloured multigraphs, that is, ones with no loops and no monochrome parallel edges.
We may also specify direction-length frameworks in an equivalent way as frameworks for the double-distance context $(\bR^2, d_b, d_a)$, where $d_b$ is Euclidean distance and $d_a$ is the angular separation distance given by
\[
d_a(p_1,p_2) = (y_1-y_2)^2/(x_1-x_2)^2.
\]
Here the sparsity type is that of $(2,2)$-tight graphs in which any monochrome subgraph is $(2,3)$-sparse. 
\medskip
 
The rigidity matrix $Df_G(p)$ for  the double-distance framework $(G,p)$  has columns indexed by the coordinates $x_i, y_i$ of the joints $p_i$, for $i=1, \dots , n=|V|)$. Each row is determined by a framework bar $\{p_i,p_j;c\}$, where $c$ is $b$ for a Euclidean bar and $a$ for a direction bar. A row for a Euclidean bar has the form,

\[
\begin{blockarray}{cccccccccc}
\begin{block}{c[ccccccccc]}
2 &...&  x_i-x_j &... & y_i-y_j &... & x_j-x_i &...  & y_j-y_i &... \\
\end{block}
\end{blockarray}
 \]
where the unspecified entries are zero, while the row for a direction bar has the form
\[
\begin{blockarray}{cccccccccc}
\begin{block}{c[ccccccccc]}
\frac{2(y_i-y_j)}{(x_i-x_j)^2} & ...&  -\frac{(y_i-y_j)}{(x_i-x_j)} &... & 1 &... & +\frac{(y_i-y_j)}{(x_i-x_j)} &...  & -1 &... \\
\end{block}
\end{blockarray}.
 \]

\medskip

3.
%A simple flat manifold analogue, from a combinatorial point of view, is the realisation of the $2$-torus as the usual identification space $\T= [0,1]^2/\sim$ endowed with the flat geodesic distance. This leads to the essentially smooth double-distance context $(\T, \T_0, d_b, d_r)$ and  $(2,2)$-sparsity since there is a $2$-dimensional space of rigid motion infinitesimal flexes. This setting is closely related to the context of $2$-periodic bar-joint frameworks in the plane. See also Whiteley \cite{Whiunion} and Ross \cite{Ross}. 
Consider the torus identification space $\T= [0,1]^2/\sim$ endowed with the flat geodesic distance constraint. By this we mean the minimum of the direct Euclidean distance $d_b(p_1, p_2)$ and the
"re-entrant" distance $d_r(p_1, p_2)$ defined as the minimum of the 2  linear distances associated with the "linear" paths through the horizontal and vertical boundaries of $[0,1]^2$.
 This  leads to the essentially smooth double-distance context $(\T, \T_0, d_b, d_r)$, for direct and re-entrant distances, where $\T_0= (0,1)^2$.
Noting that there is only a $2$-dimensional space of rigid motion infinitesimal flexes, coming from translations, the relevant mixed sparsity condition is $(2,2)$-sparseness with $(2,3)$-sparseness for blue subgraphs.
This setting is closely related to the context of $2$-periodic bar-joint frameworks in the plane. See also Whiteley \cite{Whiunion} and Ross \cite{Ross}. 
\medskip

\section{Double-distance frameworks on surfaces}\label{s:cylinder}

We consider here the unit sphere $\S$ in $\bR^3$ defined by the equation $x^2+y^2+z^2=1$ and the unit cylinder $\Y$ defined by $x^2+y^2=1$.
One may also consider extensions to families of concentric surfaces but we will not do so here.

In \cite{nix-owe-pow-1} it was shown that a framework realised generically on $\M$ with Euclidean distance constraints is minimally rigid if and only if the underlying graph is $(2,k)$-tight, where: $\M=\S$ and $k=3$, or where $\M=\Y$ and $k=2$.
Analogous results were obtained by Whiteley \cite{Whiunion} for the case when only geodesic constraints were considered.
In this section we will extend these results to characterise minimally rigid frameworks on $\M$ when geodesic and Euclidean constraints are present simultaneously. Such frameworks are frameworks $(G,p)$ for a smooth double-distance context $(\M, d_b, d_g)$ in the sense given in Section 2.
The graph edge set is the disjoint union $E_b\cup E_r$ where the blue edges in $E_b$ represent constraints implied by Euclidean distances in $\bR^3$ and the red edges of $E_r$ represent geodesic  distances. Note that in this smooth manifold case a velocity vector 
$u:V \to \bR^2$ for a framework $(G,p)$ is a selection of vectors $u(v)$ in the tangent space of $p(v)$ for each vertex
$v$. 
 
Whiteley \cite{Whicones} showed that, in the context of rigidity, there is an equivalence between geodesic bar constraints and Euclidean bar constraints on the sphere. In our formalism we note that for the sphere the infinitesimal flex space for the Euclidean bar $\{p_1,p_2;d_b\}$ is {identical} to the  infinitesimal flex space for the geodesic bar $\{p_1,p_2;d_g\}$. 
From these observations we obtain the following theorem.

\begin{thm}
Let $G$ be a bi-coloured graph and let $(G,p)$ be a completely regular framework on $\S$ for Euclidean and geodesic distances.
Then the framework $(G,p)$ is minimally rigid on $\S$ if and only if $G$ is $(2,3)$-tight.
\end{thm}

For other surfaces the distinction between geodesic and Euclidean constraints leads to differing combinatorics and we now consider this in the case of the cylinder. The \emph{rigidity matrix} $R_\Y(G,p)$ for a framework $(G,p)$ on the cylinder $\Y$ is the $|E|+|V|\times 3|V|$ matrix where the first $|E|$ rows correspond to the usual rigidity matrix $R(G,p)$ {(with its rows partitioned into the two types of constraint), and the final $|V|$ rows form a block diagonal matrix with non-zero entries given by normal vectors to the cylinder (as in \cite{nix-owe-pow-1,nix-owe-pow-2})}. The rank of $R_\Y(G,p)$ is at most $3|V|-2$ and, {by Theorem \ref{t:rigidityequivalence}}, we may say that $(G,p)$ is \emph{rigid on $\Y$} when the rank is exactly $3|V|-2$.
 
We next state the main theorem of Section \ref{s:cylinder}.

\begin{thm}\label{thm:cylinder}
Let $G$ be a bi-coloured multigraph.
A completely regular framework $(G,p)$ is minimally rigid on $\Y$ if and only if $G$ is $(2,2)$-tight.
\end{thm}

\subsection{$(2,2)$-tight bi-coloured graphs}\label{s:22tightbicoloured}

Our key tool, which is of independent interest, is an inductive construction of bi-coloured $(2,2)$-tight graphs. 
%Note that this combinatorial result requires more operations than the corresponding result we used in Section \ref{sec:p2}. This will allow us to use only a weaker geometric proof for 1-extension where the colours of the new graph are determined by the initial graph. 
Note 
%also 
that we cannot immediately use recursive constructions of $(2,2)$-tight multigraphs, or simple graphs, appearing in the literature \cite{nix-owe-pow-1,Tay} because of the requirement that each monochrome subgraph is simple. 

First we define the relevant coloured graph construction moves.

%{\color{blue}[0- and 1-extension definitions added.]}

A \emph{0-extension} (or Henneberg 1 move) adds a vertex $v$ and two edges incident to $v$. The new vertex $v$ can have two neighbours (with the two edges having arbitrary colours), or one neighbour (with one edge of each colour). 
We refer to the reverse operation as a \emph{0-reduction}.

 A \emph{1-extension} (or Henneberg 2 move) deletes an edge  $xy$ and adds a new vertex $v$ of degree 3 which is adjacent to $x$ and $y$. The third edge incident to $v$ can be incident to $x$ if the two parallel edges $xv$ have distinct colours. We refer to the reverse operation as a \emph{1-reduction}.

We say that a 1-extension deleting the edge $xy$ and adding a new vertex $v$ is \emph{colour-restricted} if the colour $c$ of $xy$ is the colour of $xv$ and of $yv$. In this section we only apply colour limited 1-extensions that result in a new degree 3 vertex with three distinct neighbours.

Define a \emph{colour-restricted vertex split} to be {}the following restricted version of} the usual vertex splitting operation {\cite{W}}. This move splits a vertex $v$ into two vertices $v_1,v_2$, adds the edge $v_1v_2$, chooses a neighbour $x$ of $v$ and replaces the edge $xv$ by two edges $xv_1,xv_2$ and then replaces all other edges $yv$ with either $yv_1$ or $yv_2$. Also we require  the additional restriction that the colour $c$ of $xv$ is the colour of the three edges $v_1v_2,xv_1,xv_2$. We refer to the reverse move as \emph{edge contraction}.

Define a \emph{graph extension} by a graph $H$ with $f(H)=2$ on a vertex $v$ to form a new graph by deleting $v$, adding $H$ disjoint from $G$ and replacing each edge $xv$ with an edge $xy$ for some $y\in H$. (We make no colour restriction on graph extensions.) We refer to the reverse move by \emph{graph contraction}. 

Note that $H$ may be a multigraph, e.g. the graph consisting of two vertices with one red and one blue edge between them. In this specific case if one of the two vertices of $H$ has degree 3 then this is a 1-extension. Hence graph extension is a generalisation of this type of 1-extension. A similar remark, with one type of 0-extension, applies when one of the vertices has degree 2. 

Let $G$ be $(2,2)$-tight. Then a graph, or edge, contraction on $G$ is allowable if the resulting graph is $(2,2)$-tight.

\begin{lem}\label{lem:contractions}
Let $G=(V,E)$ be a bi-coloured graph which is $(2,2)$-tight.\\
(a) Suppose $H$ is a proper subgraph of $G$ with $f(H)=2$. Then there is an allowable graph contraction on $H$ unless there is a vertex $u$ in $G-H$ and two vertices $a,b$ in $H$ with $(au,c),(bu,c)\in E$. \\
(b) Suppose $xy$ is an edge contained in an induced subgraph isomorphic to $K_3$ on $x,y,z$ with coloured edges $(xy,c),(xz,c),(yz,c)$. Then there is an allowable $K_3$ contraction on $xy$ unless there is a vertex $a\in V(G-K_3)$ with two edges of the same colour to the $K_3$ subgraph or there is a subgraph $Y$ with $f(Y)=2$, $x,y\in Y$ and $z\notin Y$.
\end{lem}

\begin{proof}
(a) It is easy to see that $f(G/H)=f(G)$ and that any $Y'\subset G/H$ with $f(Y)<2$ would imply a corresponding $Y\subset G$ (such that $Y/H=Y'$) with $f(Y)<2$. $G/H$ is monochrome simple unless $G/H$ contains two parallel edges of the same colour which implies the claim.

(b) Suppose the graph resulting from the edge contraction on $xy$ is denoted by $G'$. $G'$ is $(2,2)$-sparse provided it has no subgraph $Y'$ with $f(Y')<2$. $Y'$ is also a subgraph of $G$ unless it derives from a subgraph $Y\subset G$ with $xy\in Y$ and $f(Y')<f(Y)$ only if $z\notin Y$. It follows that $G'$ is $(2,2)$-tight unless it contains parallel edges of the same colour giving the result.
\end{proof}

A $(2,3)$-circuit is a (multi)graph $G$ with $|E|=2|V|-2$ and $|E'|\leq 2|V'|-3$ for any proper subgraph with at least one edge.

\begin{lem}\label{lem:plusone}
Let $G$ be $(2,2)$-tight. Then either:\\
(a) $G$ is a $(2,3)$-circuit;\\
(b) $G$ has a proper subgraph $H$ with $f(H)=2$ such that any vertex $u\in V(G)-V(H)$ is either adjacent to at most one vertex in $V(H)$ or is adjacent to two but the two edges have different colours; or\\
(c) $G$ has a vertex of degree 2.
\end{lem}

\begin{proof}
Suppose (a) fails. Then $G$ has a proper subgraph $H$ with $f(H)=2$. Let $H'\supset H$ be the smallest $(2,2)$-tight graph containing $H$ such that there is no vertex $u$ in $G-H'$ and two vertices $a,b$ in $H'$ with $au,bu$ both having the same colour. Then {since such vertices $u$ have at most $2$ edges to $H$} either (b) holds or $H'=G$. If (b) fails then $H'$ can be formed from $H$ by applying a sequence of {monochrome} 0-extensions  giving (c). 
\end{proof}

\begin{thm}\label{thm:22restricted}
A bi-coloured multigraph $G$ is $(2,2)$-tight if and only if $G$ can be generated from $K_1$ or $K_4$ (where the colouring has at least 5 edges of the same colour) by 0-extensions, colour-restricted 1-extensions, colour-restricted vertex splits and graph extensions.
\end{thm}

\begin{proof}
Let $G$ be $(2,2)$-tight. Suppose $G=K_4$. If there are at least two edges of each colour, it is easy to check that there is a colour-restricted 1-reduction on $G$ resulting in a $(2,2)$-tight graph.
Hence, by the assumptions of the theorem we may suppose $G$ is neither $K_1$ nor $K_4$.

Consider the possibilities given in Lemma \ref{lem:plusone}. If (b) holds then we can use Lemma \ref{lem:contractions}(a) to contract a proper $(2,2)$-tight subgraph. If (c) holds then we can use 0-reduction. Hence we may suppose $G$ is a $(2,3)$-circuit. It follows that $G$ has minimum degree 3 and every degree 3 vertex has 3 neighbours. Note also that, since $G$ is a $(2,3)$-circuit, $f(H)\geq 3$ for any proper $H\subset G$ with $|V(H)|\geq 2$.

Let $v$ be a degree 3 vertex with neighbours $x,y,z$. Since $G$ is a $(2,3)$-circuit distinct from $K_4$, $G$ contains no copy of $K_4$. 
Thus without loss of generality we may suppose that $xy\notin E$. 

Suppose first that the edges $xv,yv$ have the same colour $c$.
Now consider $G'=G-v+xy$ (with the colour $c$). $G'$ is $(2,2)$-tight unless $G$ contained a subgraph $H$ containing $x,y$ but not $v,z$
 with $f(H)=2$, but this would contradict the fact that $G$ is a $(2,3)$-circuit. Hence $G'$ is $(2,2)$-tight. Since $(xy,c)\notin E$, $G'$ is also monochrome simple. 

It remains to consider the case, without loss of generality, when the edges incident to $v$ are $(xv,b)$, $(yv,r)$ and $(zv,b)$. Here we consider the graph $G'=G-v+xz$.  By the argument in the previous paragraph it follows that if $G'$ is not $(2,2)$-tight then $(xz,b)\in G$. Now $\{x,z,v\}$ induces a monochrome triangle (note if $(xz,r)\in E$ then we contradict the fact $G$ is a $(2,3)$-circuit). We can now use Lemma \ref{lem:contractions}(b) and the facts that $v$ has degree 3 and $G$ is a $(2,3)$-circuit to complete the proof. 
\end{proof}

\subsection{Special position arguments}

In the next section, and in subsequent sections for other contexts, 
we show that the graph operations of the inductive construction preserve the infinitesimal rigidity of completely regular frameworks. For a 0-extension move this is elementary. However, for $1$-extension moves we shall make use of special position arguments  which are specific to the geometry of the context.

We first recall the special position argument for infinitesimal rigidity preservation for $1$-extensions in the case 
of traditional bar-joint frameworks in the plane, and in Definition \ref{d:special3joints} we introduce some convenient terminology.

Let $(G,p) \to (G',p')$ be a $1$-extension move
in which the bar $p(e_{12})$, incident to joints $p_1$ and $p_2$ in $(G,p)$, is replaced by three bars
$p'(e_{01}'), p'(e_{02}'), p'(e_{03}') $ to create $(G',p')$. Thus the framework move is determined by the removal of the bar $p(e_{12})$, the positioning of a new joint $p'_0$,  and the addition of bars from $p_0'$ to $p_1, p_2, p_3$.  In the case of Euclidean frameworks in the plane, let us say that $(G', p')$ is in a \emph{special position} if  $p_1, p_2, p_3$ are not colinear, and $p_0', p_1, p_2$ are colinear.

Suppose now that $(G,p)$ is  regular and that $(G', p')$ is in special position, as above. We show that if  $(G', p')$ is infinitesimally flexible then so too is $(G,p)$. This completes the proof, since the infinitesimal flexibility of a regular framework for $G'$ implies the infinitesimal flexibility of $(G',p')$.

By assumption the
special position framework  is infinitesimally flexible with a nonrigid motion infinitesimal flex $u'$. By the colinearity condition for the special position the velocities $u_1'$ and $u_2'$ give an infinitesimal flex $(u_1', u_2')$ of the bar $p(e_{12})$ in $(G,p)$. Thus, the restriction, $u$ say, of $u'$ to the joints of $(G,p)$ determines an infinitesimal flex of $(G,p)$. Suppose, by way of contradiction, that this is a rigid motion flex. Then there is a corresponding rigid motion flex, $u'_{\rm rig}$ say, of $(G',p')$. Subtracting this flex from $u'$ we arrive at a nonrigid motion infinitesimal flex of $(G',p')$ which has zero velocities on all joints except $p_0$. Since $p_1, p_2$ and $p_3$ are not colinear, by the special position condition, this is a contradiction, as required.

Note that the essential part of this argument is that starting with a typical bar $p_1p_2$ we are able to find a special position for a third joint $p_0$ with the property that {any} infinitesimal flex of the 2-bar framework with bars $p_0p_1, p_0p_2$ is automatically an infinitesimal flex of the bar $p_1p_2$. This motivates the following definition for a general double-distance context.
%We indicate a bar with joints $p_1, p_2$ for the constraint distance $d_c(\cdot, \cdot)$ by the pair $(p_1p_2;c)$. Similarly we may specify a triangle framework  by an ordered triple $p_0p_1, p_0p_2, p_1p_2$ together with its colour list $\{c_{01}, c_{02}; c_{12}\}$.

\begin{defn}\label{d:special3joints} Let $\T$ be a  bar-joint framework triangle  for the double-distance context $(X, d_b, d_r)$, with distinct joints $p_0, p_1, p_2$ and bars $p_{ij}$ between  $p_i, p_j$. Then the pair $(\T, p_0)$ is in special position if every infinitesimal flex of the $2$-bar subframework with bars $p_{01}, p_{02}$  is also an infinitesimal flex of the bar $p_{12}$.
\end{defn}

In many cases we shall be able to use a special position infinitesimal flex argument as above, or a limiting variant of this, as in \cite[Lemma 4.2]{nix-owe-pow-2}, to obtain rigidity preservation. However, as we see in Section \ref{s:EuclideanAndNon}, for some mixed colour cases, such as Lemma \ref{l:coloured1extensionSpecialCase} for example, we need a more involved special position argument.

\subsection{Geometric operations on $\Y$}\label{ss:geometricoperations}

\begin{lem}\label{lem:cr1ext}
Let $G'$ be a colour-restricted 1-extension of a graph $G$ and suppose that $(G,p)$ is a completely regular minimally rigid framework on $\Y$. Then any completely regular realisation of $G'$ is minimally rigid on $\Y$.
\end{lem}

\begin{proof}
As noted earlier we need only consider colour-restricted 1-extensions resulting in 3 distinct neighbours for the new vertex $v$.

In the case that the replaced edge $xy$ is red, corresponding to a geodesic distance, we note that any placement $p(v)$ of $v$ on the geodesic from $p(x)$ to $p(y)$ gives a 
special position triangle framework for geodesic distances. Indeed, geodesic distances correspond to planar Euclidean distances on unwrapping the cylinder to its planar covering space. Rigidity preservation now follows, for either colour for $zv$, by the usual special position argument \cite{nix-owe-pow-1}.

Suppose now that $xy$ is blue. 
Then \cite[Lemma 4.2]{nix-owe-pow-2} applies.
\end{proof}

\begin{lem}\label{lem:crvs}
Let $G'$ be a colour-restricted vertex split of a graph $G$ on a blue edge and suppose that $(G,p)$ is a completely regular minimally rigid framework on $\Y$. Then any completely regular realisation of $G'$ on $\Y$ is minimally rigid on $\Y$.
\end{lem}

\begin{proof}
Since the operation is blue colour-restricted the argument given in \cite[Lemma 5.1]{nix-owe-pow-2} can be applied  immediately.
\end{proof}

\begin{lem}\label{lem:gext}
Let $H$ be minimally rigid on $\Y$, let $G$ be a graph extension of $G/H$ and suppose that $(G/H,p|_{G/H})$ is a completely regular minimally rigid framework on $\Y$. Then any completely regular framework $(G,p)$ on $\Y$ is minimally rigid.
\end{lem}

\begin{proof}
We can directly apply the proof from \cite[Lemma 5.2]{nix-owe-pow-1}. 
\end{proof}

\subsection{Proof of Theorem \ref{thm:cylinder}}

We can now prove the main result of this section.

\begin{proof}[Proof of Theorem \ref{thm:cylinder}]
It follows from Theorem \ref{t:topcount} that any minimally rigid framework on $\Y$ is $(2,2)$-tight. For the converse we use induction on $|V|$. It is easy to see that $K_1$ and $K_4$ are minimally rigid. By Theorem \ref{thm:22restricted} it remains to show that 0-extensions, colour-restricted 1-extensions, colour-restricted vertex splits and graph extensions preserve minimal rigidity. This was the content of Lemmas \ref{lem:cr1ext}, \ref{lem:crvs} and \ref{lem:gext}.
\end{proof}

\begin{rem}{\rm 
 We remark that a second proof scheme for this theorem  could make use of a simpler combinatorial construction scheme with $1$-extension moves of unrestricted colour combination. %{\color{blue} We could state it: presumably just colour unrestricted  1-extensions and K4 additions ?} 
 On the other hand rigidity preservation for these hybrid moves would require a further case-by-case analysis. 
 %The analysis echoes the arguments in the next section in the case of Euclidean and non-Euclidean distances.
}
\end{rem}

\section{Euclidean and non-Euclidean constraints and mixed $(2,2)$- and $(2,3)$-sparsity}\label{s:EuclideanAndNon}

We now consider frameworks in $\mathbb{R}^2$ where some bars constrain the Euclidean distance between their joints and the remaining bars constrain a non-Euclidean distance for the norm $\|\cdot\|_q$, for some $q\neq 2$.
Once again the structure graph is viewed as a bi-coloured multigraph, without loop edges,  where the blue edges correspond to the Euclidean distance constraints and the red edges correspond to the remaining distance constraints. 

We shall assume that $q$ is fixed, with  $1<q<\infty, q\neq 2$, and so the general results  in Section \ref{s:doubledistance} apply  to $(\bR^2, d_b, d_r)$, where
\[
d_b(p_1, p_2) = 
 \|p_1- p_2\|_2,\quad
d_r(p_1, p_2) =  \|p_1- p_2\|_{q}.
\]
In particular a completely regular double-distance framework $(G,p)$ is infinitesmally rigid if and only if the rigidity matrix has rank $2|V|-2$.

\subsection{Mixed $(2,2)$- and $(2,3)$-sparsity}
We say that a bi-coloured graph is \emph{$(2,3)$-limited} if any subgraph with only blue edges is $(2,3)$-sparse. 
It is worth commenting that this definition is not symmetric in red and blue.
We now derive a recursive construction of bi-coloured graphs that are $(2,2)$-tight and $(2,3)$-limited.
To simplify the requirements for rigidity preservation in the next section, we consider some restriction on 1-extensions. The first of these is that when we subdivide a red edge $xy$, at least one the two new edges $xv,yv$ is coloured red. In other words we do not use the 1-extension move for the colour case $r \to \{b,b\}$. 

\begin{lem}\label{lem:22lim}
A bi-coloured multigraph $G$ is $(2,2)$-tight and $(2,3)$-limited if and only if $G$ can be generated from $K_1$ by 0-extensions and 1-extensions which are not of the colour case $r \to \{b,b\}$.
\end{lem}

\begin{proof}
Let $G=(V,E)$ be $(2,2)$-tight and $(2,3)$-limited.
Since $|E|=2|V|-2$, there exists a vertex $v$ of degree 2 or 3. If $v$ has degree 2 it is easy to see that $G-v$ is $(2,2)$-tight and $(2,3)$-limited. So we may suppose the minimum degree is 3. Let $v$ have degree 3 and 
%let $G-v+xy$ be 
consider the two graphs $G'$ resulting from a 1-reduction at $v$ adding a coloured edge on $x, y$. We see that

(i) $G'=G-v+(xy,b)$ is an allowable $1$-reduction unless $G'$ is not (2,2)-sparse or there is a blue subgraph violating (2,3)-sparsity.

(ii)  $G'=G-v+(xy,r)$ is an allowable $1$-reduction unless $G'$ is not (2,2)-sparse or $(xy,r)\in E$.

% then we are done unless (a) $(xy,r)\in E$; (b) $G-v+xy$ is not $(2,2)$-sparse; or (c) $G-v+xy$ has a blue subgraph which violates $(2,3)$-sparsity. {\color{red}I think the logic of this needs to be clearer - do we read (a) or (b) or (c)? and what is the colour of $xy$ in (b) and (c)? }

Suppose first that $v$ has exactly two neighbours $x$ and $y$. If, in either case (i) or (ii), $G'$ has a subgraph which fails $(2,2)$-sparsity  then it follows that there is a subgraph of $G$ which violates $(2,2)$-sparsity, a contradiction.
Suppose that (i) holds and there is a blue subgraph violating (2,3)-sparsity. This implies that $G$ has a subgraph $H$ containing $x,y$ but not $v$ with $f(H)=3$. Since $f(H\cup \{v\})\leq 2$ we have that $(xy,r)\notin E$ and so, by (ii) we are done. Since there is at most 1 edge between $x, y$ we have shown that an allowable $1$-reduction is always possible in the two neighbour case. 

Now suppose that $v$ has 3 neighbours $x,y,z$. First suppose at most one of the edges incident to $v$ is blue. (Since we only exclude the colour case $r \to \{b,b\}$, this is sufficient to allow us to consider all six possible coloured 1-reductions at $v$.) 

Suppose that $(2,2)$-sparsity fails for $x,y$ and for $y,z$. Then there is a subgraph $H_{xy}$ with $f(H_{xy})=2$ containing $x,y$ but not $z,v$
 and a subgraph $H_{yz}$ with $f(H_{yz})=2$ containing $y,z$ but not $x,v$. Then $f(H_{xy}\cap H_{yz})\geq 2$ so it follows that $f(H_{xy}\cup H_{yz}\cup \{v\})<2$ violating $(2,2)$-sparsity.
Hence $(2,2)$-sparsity fails for at most one pair of neighbours of $v$. Suppose that there is a blue $(2,3)$-tight subgraph $F$ containing $x,y,z$ but not $v$. Then there are no red edges induced by the vertex set of $F$. Since $(2,2)$-sparsity fails for at most for one pair  we see there is a valid 1-reduction at $v$.

Next suppose $(2,2)$-sparsity fails precisely for the pair $x,y$ but no other pair. 
Now let $F_{yz}$ be either a single red edge or a $(2,3)$-tight blue subgraph containing $y,z$ but not $x,v$. In either case $f(F_{yz})=3$ and $f(H_{xy}\cup F_{yz})\leq 3$. It follows that if either a single edge $(xz,r)$ or a $(2,3)$-tight blue subgraph containing $x,z$ but not $y,v$ exists then adding $v$ will give a subgraph that contradicts $(2,2)$-sparsity.

Lastly suppose that there is no $(2,2)$-sparsity failure for any pair of neighbours. Then if there is no 1-reduction to a $(2,2)$-tight and $(2,3)$-limited graph we have $(xy,r),(yz,r)\in E$ and we have subgraphs $F_{xy},F_{yz}$ as before. Since $f(F_{xy}+(xy,r))=2$ and $f(F_{yz}+(yz,r))=2$ we can use the argument above with $F_{xy}+(xy,r)$ and $F_{zy}+(yz,r)$ in place of $H_{xy}$ and $H_{yz}$ to reach a contradiction.

Hence we may suppose at least two of the edges incident to $v$ are blue. If exactly two then the restriction on colouring rules out only one of the six possible 1-reductions at $v$. Therefore we can use the same argument. 

Thus suppose all three edges incident to $v$ are blue. As above $(2,2)$-sparsity may fail for at most one pair. If this is the case then the argument above can be re-run. So suppose there is no $(2,2)$-sparsity failure for any pair. Then there is a 1-reduction on $v$ to a smaller $(2,2)$-tight and $(2,3)$-limited graph unless for all three pairs
there is a blue subgraph violating (2,3)-sparsity. This gives subgraphs $F_{xy},F_{xz},F_{yz}$ defined as above each with $f(\cdot)=3$. Suppose that $|F_{xy}\cap F_{xz}|\geq 2$. Then $f(F_{xy}\cap F_{xz})\geq 3$ and $f(F_{xy}\cup F_{xz}\cup \{v\})\leq 2$ contradicting the fact that $G$ is $(2,3)$-limited. Hence we may assume that $|F_{xy}\cap F_{xz}|=|F_{xy}\cap F_{yz}|=|F_{xz}\cap F_{yz}|=1$. It follows that $f(F_{xy}\cup F_{xz}\cup F_{yz}\cup \{v\})=2$, again contradicting the fact that $G$ is $(2,3)$-limited.
\end{proof}

We next modify the construction scheme above to avoid 1-extensions of colour case $b\to \{r,r\}$ when the blue edge has a parallel red edge. We say that this move has parallel edge colour case $b\to \{r,r\}$. We  show that this is possible at the expense of introducing an elementary substitution move which replaces 
the 3-vertex graph $K_2 \sqcup K_2$, obtained by joining 2 copies of bi-coloured $K_2$ at a common vertex, by a
red subgraph $K_4$, which we denote as $K_4^r$. This replacement is joined at the 3 vertex positions of the removed graph, denoted $x,y,z$ below, and provides a degree 3 vertex $v$ for the new graph (see Figure \ref{f:submove}).

\begin{center}
\begin{figure}[ht]
\centering
\begin{tikzpicture}[scale=1]

\filldraw (1,2) circle (2pt) node[anchor=east]{$x$};
\filldraw (0,0) circle (2pt) node[anchor=east]{$y$};
\filldraw (1,-2) circle (2pt) node[anchor=east]{$z$};

\draw
(1,2) -- (0,0) -- (1,-2);

\draw[dashed] plot[smooth, tension=1] coordinates{(1,2)(1,1) (0,0)};

\draw[dashed] plot[smooth, tension=1] coordinates{(1,-2)(1,-1) (0,0)};

\draw
(2,0) -- (3.5,0);

\draw
(3.3,.2) -- (3.5,0) -- (3.3,-.2);

\filldraw (6,2) circle (2pt) node[anchor=east]{$x$};
\filldraw (5,0) circle (2pt) node[anchor=east]{$y$};
\filldraw (6,-2) circle (2pt) node[anchor=east]{$z$};
\filldraw (8,0) circle (2pt) node[anchor=west]{$v$};

\draw
(6,-2) -- (6,2) -- (5,0) -- (6,-2) -- (8,0);

\draw
(6,2) -- (8,0) -- (5,0);

\end{tikzpicture}
\caption{The substitution operation replacing $K_2\sqcup K_2$ with $K_4^r$.}
\label{f:submove}
\end{figure}
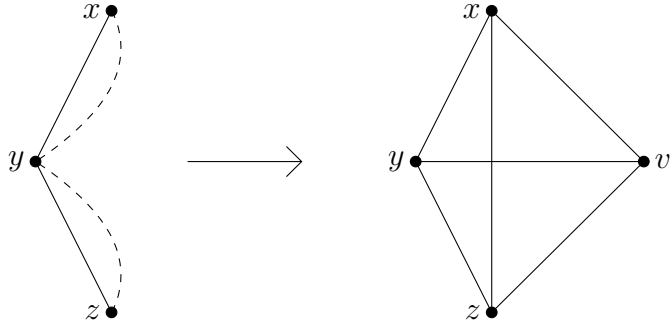
\end{center}

It is straightforward to show that this (extension) move, from $G'$ to $G$, preserves the class of $(2,2)$-tight $(2,3)$-limited graphs. It is also true that the inverse preserves these properties as we now show. Any subgraph $H$, of $G'$, violating $(2,2)$-sparsity contains either $x,y$ or $y,z$ or $x,y,z$ (otherwise it was present as a subgraph of $G$). In the latter case we get a violation in $G$ by adding $(xz,r)$ and $v$'s 3 incident edges. If $H$ contains $x,y$ but not $z$ then, in $G$, adding $z$ and $v$ and their 5 incident edges gives a violation in $G$.

Similarly, any blue subgraph $H$ violating $(2,3)$-sparsity contains either $x,y$ or $y,z$ or $x,y,z$. First the latter case, to get a violation in $G$ we subtract $(xy,b), (yz,b)$ from $H$ but add $(xy,r), (yz,r), (xz,r)$ and $v$'s 3 incident edges, the new graph has both colours but violates $(2,2)$-sparsity. Lastly if $H$ contains $x,y$ but not $z$ then, in $G$, $H+(xy,r)$ plus $z,v$ and their 5 incident edges gives a violation of $(2,2)$-sparsity.

\begin{lem}\label{l:construction2}
A bi-coloured multigraph $G$ is $(2,2)$-tight and $(2,3)$-limited if and only if $G$ can be generated from $K_1$ by $K_2\sqcup K_2$ substitutions, 0-extensions and 1-extensions which are not of the colour case $r \to \{b,b\}$ or the parallel edge colour case $b\to \{r,r\}$.
\end{lem}

\begin{proof} 
Let $G$ be $(2,2)$-tight and $(2,3)$-limited.
Let $v$ be a degree 3 vertex in $G$. 
By the discussion above, it suffices to show that if there is no allowable $1$-reduction for $v$ then $v$ lies in a $K_4^r$ subgraph.

Let $x,y,z$ be the neighbours of $v$.
We may suppose that $(xy,r)$ exists and the edges incident to $v$ from $x$ and $y$ are red. Since there is no 1-reduction to a red edge $xz$  it follows that $(xz,r)$ exists or there is a (2,2)-tight subgraph $H_{xz}$ as before (ie. containing $x,v$ but not $y,z$). Similarly, since there is no 1-reduction to a red edge $yz$ then either
$(yz,r)$ exists or there is a (2,2)-tight subgraph $H_{yz}$.
Since the subgraph $(H_{xz}\cup H_{yz})+ xy$ violates $(2,2)$-sparsity we may assume, relabelling if necessary, that $(xz,r)$ exists.

Now note that we cannot have a (2,2)-tight $H_{yz}$, since adding
the 3 edges to $v$ and the 2 red edges violate $(2,2)$-sparsity. So we can assume that we have the red edge $(yz,r)$, and hence a $K_4$ in $G'$ all of whose edges, except perhaps $vz$ are red. 

Finally, we show that $vz$ is not a blue edge. Note that the $K_4$ subgraph ensures that
$(xz, b)$ and $(yz,b)$ are not in $G$.
Since there is no allowable reduction to $(xz, b)$ or to $(yz,b)$  there are $(2,3)$-tight blue graphs $F_{xz}$ and $F_{yz}$ as before. The union of these graphs with the $K_4$ graph violates $(2,2)$-sparsity and so the proof is complete.
\end{proof}

\subsection{$1$-extensions and rigidity preservation}
\label{subsec:1ext}

We next show that various  mixed colour  1-extension moves preserve infinitesimal rigidity for completely regular frameworks in the double-distance context $(\bR^2, \|\cdot \|_2, \|\cdot \|_q)$. 
Recall that for a $1$-extension move
\[
G \to G' = G- v_1v_2 + (v_0v_1 +v_0v_2+v_0v_3)
\]
the edge $v_1v_2$ is replaced by the replacement edges $v_0v_1, v_0v_2$
and a connecting edge $v_0v_3$.  There are a number of colour cases
$c_1 \to \{c_2, c_3\}$ which we shall consider where, as before, 
$c_1$ is the colour of $v_1v_2$ and $c_2, c_3$ are the colours of the replacement edges.

\begin{lem}\label{lem:coloured1extension} 
Let $G$ and $G'$ be $(2,2)$-tight  bi-coloured multigraphs and
let $G \to G'$ be a 1-extension move, with colour case
$r\to\{b,r\}, b\to\{b,r\}, b\to\{b,b\}$ or $r\to\{r,r\}$.   
%where the colour case is not $r\to \{b, b\}$. 
If $(G,p)$ and $(G',p')$ are completely regular and $(G,p)$ is minimally rigid then $(G',p')$ is infinitesimally rigid.
\end{lem}

\begin{proof}
We may assume that $p = (p_1,\dots , p_n)$ and $p' = (p_0, p)$. We show that if $(G', p')$ is not infinitesimally rigid then the same is true for $(G, p)$. We use a limiting argument to construct an infinitesimal flex of $(G,p)$, which is in the style of the proof of  \cite[Lemma 4.2]{nix-owe-pow-2}. 

(i) Consider first the colour case $r \to \{r, b\}$ together with the assumption that $v_0v_3$ is blue.

Let $b$ be a unit vector which, at $p_2$, is tangent to the closed curve of points $(x,y)$ whose non-Euclidean distance from $p_1$ is $d_r(p_1, p_2)$. Equivalently $b$ is a vector such that the real scalar multiples $\lambda (0, b)$  provide all the real infinitesimal flexes of the form  $(0, u_2)$ for the bar $\{p_1, p_2; r\}$ in $(G,p)$ with zero velocity at $p_1$.
Let $p_0^k$ be a sequence of specialisations of $p_0$, with $p_0^k- p_2$  orthogonal to $b$, which converge to $p_2$ as $k$ tends to infinity. See Figure \ref{f:1extn_mixed}.

\begin{center}
\begin{figure}[ht]
\centering
\begin{tikzpicture}[scale=1]

\filldraw (0,0) circle (2pt) node[anchor=north]{$p_0^k$};
\filldraw (-.2,2) circle (2pt) node[anchor=west]{$p_1$};
\filldraw (-2,.4) circle (2pt) node[anchor=east]{$p_3$};
\filldraw (1,-1) circle (2pt) node[anchor=north]{$p_2$};

\draw
(-2,.4) -- (0,0) -- (1,-1) -- (1.5,-.5);

\draw
(1.4,-.5) -- (1.5,-.5) -- (1.5, -.6);

\draw
(.5,.5) -- (0,0);

\draw
(.4,.5) -- (.5,.5) -- (.5, .4);

\draw[dashed]
(0,0) -- (-.2,2);

\draw
(1.2,-.8) -- (1,-.6) -- (.8,-.8);

\end{tikzpicture}
\caption{The colour case $r \to \{b, r\}$}
\label{f:1extn_mixed}
\end{figure}
\end{center}

It follows from the assumption that there is an infinitesimal flex 
\[
u^k = (u^k_0, u^k_1, u^k_2, \dots , u^k_n)
\] 
of
$(G', p^k)$, where $p^k= (p_0^k, p)$, such that $u^k_1=0$ and  
$\|u^k\|_2=1$. Moreover, by replacing the sequence $(u^k)$ by a subsequence we may assume that this sequence has a limit
\[
u = (u_0, u_1, u_2, \dots , u_n).
\] 
This vector, which has $u_1=0$ and  unit norm, is necessarily an infinitesimal flex of the degenerate framework $(G', p^\infty)$ where $p^\infty$ is the limit of the sequence $(p^k)$. (This framework has a zero length bar whose endpoints are the coinciding placements of $v_0$ and $v_2$.)

Note that $u^k_0$ is a scalar multiple of the tangent vector $b_k$ at $p^k_0$ for the boundary curve of the non-Euclidean  ball centred at $p_1$ with radius $d_r(p_1,p_0^k)$. Taking limits we see that $u_0$ is parallel to $b$. On the other hand, considering the Euclidean bar between $p_0^k$ and $p_2$, we have $(u_0^k-u^k_2)\cdot a =0$ for all $k$, where $a$ is a nonzero vector orthogonal to $b$. 
It follows that $u_2\cdot a =0$, that is, that
 $u_2$ is a scalar multiple of $b$.

Since the vector $(0,u_2)$ can now be regarded as a flex of the non-Euclidean bar $\{p_1,p_2;r\}$ it follows that the restriction of $u$ to the joints of $(G,p)$ provides an infinitesimal flex, $u_r $ say, of $(G,p)$ with zero velocity for the joint $p_1$. Thus either $(G,p)$ is infinitesimally flexible, as desired, or $u_r=0$. However, if $u_r=0$ then, since $u$ has norm $1$ the velocity $u_0$ has norm 1 and we obtain a contradiction by considering the regular subframework of $(G', p^\infty)$ consisting of the two bars from $p_1$ and $p_3$ to $p_0^\infty$. Note that for this last contradiction the colour of the edges is irrelevant, so the argument also covers the case that the edge $v_0v_3$ is red.

%The remaining colour cases for the replacement  $v_1v_2 \to \{v_0v_1, v_0v_2\}$ are:
\medskip

(ii) The case $b \to \{b, r\}$. This case follows by a similar limiting argument associated with the specialisations  illustrated in Figure \ref{f:1extn_mixed2}. Let $m$ be a unit vector perpendicular to $p_1 - p_2$ and choose specialisations $p_0^k$ so that $m$ is the flex direction for $p_2$ for a flex of the bar $\{p_0^k,p_2; r\}$ which has zero velocity at $p_0^k$.

\begin{center}
\begin{figure}[ht]
\centering
\begin{tikzpicture}[scale=1]

\filldraw (0,0) circle (2pt) node[anchor=west]{$p_0^k$};
\filldraw (-.8,2) circle (2pt) node[anchor=west]{$p_1$};
\filldraw (-2,.3) circle (2pt) node[anchor=east]{$p_3$};
\filldraw (-.5,-1) circle (2pt) node[anchor=north]{$p_2$};

\filldraw (0,-.8) circle (0pt) node[anchor=north]{$m$};

\draw
(-.8,2) -- (0,0) -- (-2,.3);

\draw[dashed]
(0,0) -- (-.5,-1);

\draw[dotted]
(-.5,-1) -- (-.8,2);

\draw
(0,-.8) -- (-.5,-1);

\draw
(-.2,-.8) -- (0,-.8) -- (-.1,-1);

\end{tikzpicture}
\caption{The colour case $b \to \{b, r\}$.}
\label{f:1extn_mixed2}
\end{figure}
\end{center} 

Consider a converging sequence of flexes $u^k$ of $(G',p^k)$, as before, with $u^k_1=0$ and unit norm. Then  the velocity $u_2^k$ is the sum 
$u_0^k+\lambda_k m$,
with $u_0^k$ viewed as a translation component for the bar $\{p_0^k,p_2; r\}$, and  $\lambda_km$ viewed as a `non-Euclidean rotation' component of the same bar with zero velocity at $ p_0^k$. As $k$ tends to infinity  $u_0^k$ tends to a scalar multiple of $m$ and so $u_2$ is also a scalar multiple of $m$. Thus the vector $(0,u_2)$ can be viewed as an infinitesimal flex of the bar $\{p_1, p_2;b\}$ and we can argue exactly as before. 

%{\color{red}COMMENT: this $b \to \{b,r\}$ argument seems fine and we would like a version for $\P_2$...}

(iii) The case $b \to \{b,b\}$. This  follows from the standard linear specialisation, irrespective of the colour of $v_0v_3$.

(iv) The case $r \to \{r,r\}$. This follows as in (iii) from a linear specialisation. (See also Kitson and Power \cite{kit-pow-1} for the purely non-Euclidean case.)
% (Include my earlier argument in an Appendix ? or here ?)
\end{proof}

%{\color{red}FURTHER adjustments to do from here}
 
%Bearing in mind that our combinatorial construction does not require the 1-extension with colour case $r \to \{b, b\}$ it remains to consider the case $ b \to \{r,r\}$. 

The proof of the next lemma is analogous to the continuous flex special position argument in \cite{nix-owe-pow-1} for 1-extensions for frameworks on the cylinder.
%We first look at a special case of this.

\begin{lem}\label{l:coloured1extensionSpecialCase} 
Let $n\geq 2$ and suppose that all
$(2,2)$-tight, $(2,3)$-limited,  bi-coloured multigraphs with $n$ vertices have completely regular placements which are minimally infinitesimally rigid. Let $G$ be such a graph
and let $G \to G'$ be a 1-extension move on the edge $e=(v_1v_2,b) $ for the colour case $b\to\{r,r\}$, where $G$ has no edge $(v_1v_2, r)$. Then the completely regular placements of $G'$ are minimally infinitesimally rigid.
\end{lem}

\begin{proof}
%Suppose first that $G$ does not contain $(v_1v_2,r)$. 
Let $G^r$ be $G$ with $(v_1v_2,b)$ replaced by $(v_1v_2,r)$. Note that 
$G^r$ is also $(2,2)$-tight and $(2,3)$-limited.
Let $p$ be a placement of the vertices of $G$ such that both
$(G,p)$ and $(G^r,p)$ are completely regular and infinitesimally rigid.
Consider the framework $(G\backslash e, p)$ obtained by removing the  Euclidean bar $\{p_1, p_2;b\}$.  When $p_1$ is fixed the framework $(G\backslash e, p)$ has one degree of flexibility, roughly speaking. 
To be precise, 
%there exists a continuous flex $p(t), 0\leq t \leq 1$, such that $p_1(t)=p_1$ for all $t$ and $|p_1-p_2(t)|$ is a strictly decreasing real analytic function. This follows in a standard way from the implicit function theorem for real analytic functions and the real analyticity of the constraint functions $d_b$ and $d_r$ on $\bR^2 \times \bR^2$.
the space of infinitesimal flexes $u$ of this framework which have local velocity $u_1=0$ at $p_1$ is $1$-dimensional. Also $u$ is nonzero if and only if the local velocity $u_2$ is nonzero and in this case $u_2$ is a nonzero multiple of a unit vector, $a$ say, at $p_2$. In fact $a$  is not orthogonal to $p_1-p_2$ since $(G,p)$ is infinitesimally rigid. We shall make use of the companion fact that since $(G^r,p)$ is infinitesimally rigid the vector $a$ is not parallel to a vector, $b$ say, at $p_2$ which is tangent to the curve of points $z=(x,y)$ with $d_r(z,p_1) = d_r(p_2,p_1)$. 

Choose a joint $p_{n+1}$ on the interior of the line segment $[p_1, p_2]$ and add bars $\{p_1,p_{n+1};r\}$ and $ \{p_{n+1}, p_2;r\}$, as indicated in Figure \ref{f:continuous}, to create a framework $((G\backslash e)^+, p^+)$.

\begin{center}
\begin{figure}[ht]
\centering
\begin{tikzpicture}[scale=1]

%$p_{n+1}$

\filldraw (-.4,.3) circle (2pt) node[anchor=west]{};
\filldraw (0,2) circle (2pt) node[anchor=south]{$p_1$};
\filldraw (-1.6,-.3) circle (2pt) node[anchor=east]{$p_3$};
\filldraw (-.7,-1) circle (2pt) node[anchor=north]{$p_2$};

\draw
(-.5,2.4) -- (0,2) -- (.4,2.4);

\draw
(-1.2,-1.4) -- (-.7,-1) -- (-.3,-1.4);

\draw[dashed]
(0,2) -- (-.7,-1);

\draw[dashed]
(-1,.4) -- (0.2,.2);

\draw
(-.9,.25) -- (-1,.4) -- (-.85,.5);

\draw
(.05,.1) -- (.2,.2) -- (.1,.35);

\end{tikzpicture}
\caption{The construction of $((G\backslash e)^+, p^+)$ from $(G,p)$.}
\label{f:continuous}
\end{figure}
\end{center}

Let $u$ be a nonzero infinitesimal flex of $((G\backslash e)^+, p^+)$ with $u_1=0$. By the choice of $p_{n+1}$, since $u_1=0$ the local velocity of $u_{2}$ is in the direction $a$, as well as direction $b$, and so is equal to $0$. Since $(G^r,p)$ is infinitesimally rigid the local velocities
$u_3, \dots , u_n$ must also be zero. Thus $u_{n+1}\neq 0$. 
%(Informally, if  $u_2\neq 0$  then  $u_{n+1}$ would have to be infinite.)
Finally construct the framework $(G',p')$ by adding the bar $\{p_3,p_{n+1};c\}$ to $((G\backslash e)^+, p^+)$.
The completely regular placements $p$ above (for $G$ and $G^r$) are dense so we can arrange that  $p_1, p_2, p_3$ are not colinear. It follows  that if $c$ is red then an infinitesimal flex $u$ of $(G,p)$ fixing $p_1$, as above, does not extend to an infinitesimal flex of $(G',p')$. Thus $(G', p')$ is infinitesimally rigid, and so any completely regular framework for $G'$ is infinitesimally rigid.
Similarly, if $c$ is blue we may choose $p_{n+1}$ on the line segment  such that $p_3-p_{n+1}$ is not orthogonal to the vector $b$ to arrive at the same conclusion.
%ATTEMPT Suppose now that $G$ contains  $(v_1v_2,r)$. Then the completely regular framework $(G,p)$ contains the red bar $\{p_1, p_2;r\}$. Consider  $p'=(p,p_{n+1})$ with completely regular framework$(G', p')$ with $G'$ as before. Suppose, by way of contradiction, that there is an infinitesimal flex $u$ for $(G', p')$ with $u_1=0$. In view of the bar $\{p_1, p_2;r\}$ the local velocity $u_2$ is a multiple of the unit vector $b$ with $b$ as before, namely the flex direction for this red edge when $u_1=0$. Without loss of generality we can normalise
\end{proof}

\subsection{Minimally rigid frameworks for $(\bR^2, \|\cdot\|_2, \|\cdot\|_q)$}

\begin{thm}\label{t:euclidean_noneuclidean}
Let $G$ be a bi-coloured multigraph.
A completely regular framework $(G,p)$ in $(\bR^2,(\|\cdot\|_2,\|\cdot\|_q))$, where $1<q<\infty, q\neq 2$, is minimally infinitesimally rigid if and only if $G$ is $(2,2)$-tight and $(2,3)$-limited.
\end{thm}

\begin{proof}
It follows from Theorem \ref{t:topcount} that any minimally rigid framework for the context $(\bR^2,(\|\cdot\|_2,\|\cdot\|_q))$ has an underlying graph which is $(2,2)$-tight and $(2,3)$-limited. For the converse we use induction on $|V|$. The single vertex graph $K_1$ is trivially minimally rigid. 
It is straightforward to show that 0-extensions preserve rigidity. By Lemma \ref{l:construction2} it remains to show that the relevant construction moves preserve minimal infinitesimal rigidity. This is elementary in the case of the substitution move for $K_2\sqcup K_2$, since this is a case of rigid subgraph substitution. Also this holds for the required 1-extension moves  by Lemmas \ref{lem:coloured1extension}, \ref{l:coloured1extensionSpecialCase}.
\end{proof}

\section{Non-Euclidean direction-length frameworks}\label{s:directionlength}
 We now consider  double-distance frameworks in $(\bR^2, \|\cdot\|_q)$ where the red edges correspond to distance under the $q$-norm, with $q\in (1,\infty), q \neq 2$, and the blue edges correspond to direction constraints. 
%(see also \cite{SW}). 
In this setting blue subframeworks admit dilation flexes and it follows that the appropriate sparsity type is that of bi-coloured multigraphs which are $(2,2)$-tight and $(2,3)$-limited. The inductive graph construction of Lemma \ref{l:construction2} is thus available in this setting.
We note that rigidity in the case $q=2$ has been characterised combinatorially by Servatius  and Whiteley \cite{SW}.

%and the observation that 0-extensions preserve direction length rigidity in $(\bR^2, \|\cdot\|_q)$, this reduces the characterisation of direction-length rigidity in $(\bR^2, \|\cdot\|_q)$ to proving an analogue of Lemma \ref{lem:coloured1extension}.

%Once again we must place the mild restriction on 1-extensions that when we subdivide a red edge $xy$, at least one the two new edges $xv,yv$ is coloured red. 

%$q$ is fixed with  $1<q<\infty, q\neq 2$, and 
The general results in Section \ref{s:doubledistance} apply to the double distance context $(\bR^2, d_b, d_r)$, where
\[
d_b(p_1, p_2) = 
d_a(p_1- p_2),\quad
d_r(p_1, p_2) =  \|p_1- p_2\|_{q}.
\]
%and $\dir (p_1- p_2)$ is positive angle in $[0,\pi)$ subtended by the $x$-axis and the line through $p_1$ and $p_2$. There is freedom in the choice of coordinates to measure direction separation by considering an alternative reference line to the $x$-axis. Since we wish the rigidity map $f_G$ to be a differentiable function at $p$ we may alternatively measure $\dir$ with respect to a reference line which is not parallel to any bar of $(G,p)$. In this way we can assume that the rigidity map is differentiable at $p$ and that there is a well-defined rigidity matrix. 
Also, the direction-length framework $(G,p)$ is infinitesimally rigid if and only if the rigidity matrix has rank $2|V|-2$. 
%{\color{red}perhaps do tangent  as $d_b$ measure ... and do this in the needed earlier rigidity matrix example }

It is straightforward to see that a 0-extension move and the $K_2\sqcup K_2$ substitution move preserve direction-length rigidity in $(\bR^2, \|\cdot\|_q)$.
For four of the  $1$-extension colour cases we have the following analogue of Lemma \ref{lem:coloured1extension}.

% this reduces the characterisation of direction-length rigidity in $(\bR^2, \|\cdot\|_q)$ to proving an analogue of Lemma \ref{lem:coloured1extension}.

\begin{lem}\label{lem:coloured1extensionDIRNLEGTH} 
Let $G$ and $G'$ be $(2,2)$-tight  bi-coloured multigraphs and
let $G \to G'$ be a 1-extension move, with colour case
$r\to\{b,r\}, b\to\{b,r\}, b\to\{b,b\}$ or $r\to\{r,r\}$.   
%where the colour case is not $r\to \{b, b\}$. 
If $(G,p)$ and $(G',p')$ are completely regular direction-length frameworks and $(G,p)$ is minimally rigid then $(G',p')$ is minimally infinitesimally rigid.
\end{lem}

\begin{proof}
Assume that $p = (p_1,\dots , p_n)$ and $p' = (p_0, p)$. We show that if $(G', p')$ is not infinitesimally rigid then the same is true for $(G, p)$. As before we consider four separate cases according to colour. For the limiting arguments we identify a sequence of diagrams given by a variable joint $p_0^k$ which converges to a position  $p_0^\infty$ coincident with the position of $p_2$. The positioning of
$p_0^k$ is chosen so that (assuming $(G',p')$ is not rigid)
the constructed nontrivial limit flex, of the degenerate limit framework, restricted to the joints of $(G,p)$, yields a nontrivial flex of $(G,p)$.

(i) Consider first the colour case $r \to \{r, b\}$. The diagram sequence is indicated in Figure  \ref{f:nonE_r2rb} where the angle shown is constant and indicates the direction for a flex of the non-Euclidean (red) bar $p_1p_2$ with zero local velocity at $p_1$. We also assume that $p_1p_0^k$ is a non-Euclidean (red) bar and so it follows that 
%a unit norm limit flex $u$, with zero velocity $u_1=0$ at $p_1$, has a 
the local velocity vectors $u_0^k$ at $p_0^k$ have directions converging towards this angle with $p_1p_2$. Since $p_2p_0^k$ is a blue direction-constraint bar, the sequence of local velocities $u_2^k$ at $p_2$ are similarly convergent. However, the limit of this sequence is  $u_2$ and so the pair $u_1=0$ and $u_2$ gives a flex of the bar $\{p_1,p_2;r\}$. Thus the limit flex $u$ of the degenerate framework when restricted to the joints of $(G,p)$ gives a non rigid motion flex of $(G,p)$, as desired.

\begin{center}
\begin{figure}[ht]
\centering
\begin{tikzpicture}[scale=1]

\filldraw (0,0) circle (2pt) node[anchor=west]{$p_0^k$};
\filldraw (-2,2) circle (2pt) node[anchor=west]{$p_1$};
\filldraw (1.5,1) circle (2pt) node[anchor=east]{$p_3$};
\filldraw (-.6,-1.1) circle (2pt) node[anchor=north]{$p_2$};

%\filldraw (0,-.8) circle (0pt) node[anchor=north]{$b$};

\draw
(-2,2) -- (0,0) -- (1.5,1);

\draw[dashed]
(0,0) -- (-.6,-1.1);

\draw[dotted]
(-.6,-1.1) -- (-2,2);

\draw plot[smooth, tension=1] coordinates{(-.8,-.7)(-.6,-.6)(-.4,-.7)};

\end{tikzpicture}
\caption{Colour case $r \to \{r,b\}$.}
\label{f:nonE_r2rb}
\end{figure}
\end{center} 

(ii) The colour case $b \to \{b, r\}$. For each $k$ the joint $p_0^k$ is chosen on the line through $p_2$ whose angle with the line from $p_1$ to $p_2$ is such that the velocity direction at $p_2$, for a flex of the (red) length-separation bar $p_2p_0^k$ with zero velocity at $p_0^k$, is in the direction from $p_2$ to $p_1$. See Figure \ref{f:nonE_b2rb}. 

We consider, as before, a sequence $p_0^k$ which  converges to $p_2$ together with a convergent sequence of unit norm flexes $u^k$ with local velocities $u_1^k$ at $p_1$ equal to 0. 
The velocity vector $u_0^k$ is directed towards $p_1$ since the bar between the joints is a blue direction-separation bar. Since $p_2p_0^k$ is a length-separation bar we claim that it follows from the chosen geometry that the limit velocity vector $u$ has local velocity $u_2$ at $p_2$  in the direction of the line from $p_2$ to $p_1$. This completes the proof, as in (i). To see the claim note first that $u_0^k$ converges to a velocity $u_0$ which is in the direction from $p_2$ to $p_1$. Also by the chosen geometry the velocity $u_2^k$ is a sum $u_0^k+ w_2^k$ where $w_2^k$ is in the direction from $p_2$ to $p_1$, and so the claim follows.

%
%\begin{center}
%\begin{figure}[ht]
%\centering
%\includegraphics[width=5cm]{DLblue_to_redblue.eps}
%\caption{Colour case $b \to \{r,b$\}.}
%\label{f:nonE_b2rb}
%\end{figure}
%\end{center}

\begin{center}
\begin{figure}[ht]
\centering
\begin{tikzpicture}[scale=1]

\filldraw (0,0) circle (2pt) node[anchor=west]{$p_0^k$};
\filldraw (-2,2) circle (2pt) node[anchor=west]{$p_1$};
\filldraw (1.5,1) circle (2pt) node[anchor=east]{$p_3$};
\filldraw (-.6,-1.1) circle (2pt) node[anchor=north]{$p_2$};

%\filldraw (0,-.8) circle (0pt) node[anchor=north]{$b$};

\draw[dashed]
(-2,2) -- (0,0);

\draw
(1.5,1) -- (0,0) -- (-.6,-1.1);

\draw[dotted]
(-.6,-1.1) -- (-2,2);

\draw plot[smooth, tension=1] coordinates{(-.8,-.7)(-.6,-.6)(-.4,-.7)};

\end{tikzpicture}
\caption{Colour case $b \to \{b,r\}$.}
\label{f:nonE_b2rb}
\end{figure}
\end{center} 

(iii) The colour case $b \to \{b,b\}$, for purely direction  constraints  follows from the standard linear specialisation, irrespective of the colour of $v_0v_3$.
%{\color{red}perhaps more here}

(iv) The purely non-Euclidean bar colour case $r \to \{r,r\}$ follows, as in (iii), from a linear specialisation as in Kitson and Power \cite{kit-pow-1}.
% (Include my earlier argument in an Appendix ? or here ?)

\end{proof}

We have the following analogue of Lemma \ref{l:coloured1extensionSpecialCase} whose proof, which we omit, follows precisely the same logic.

\begin{lem}\label{l:coloured1extensionSpecialCaseDL} 
Let $n\geq 2$ and suppose that all
$(2,2)$-tight, $(2,3)$-limited,  bi-coloured multigraphs with $n$ vertices have completely regular direction-length frameworks which are minimally infinitesimally rigid. Let $G$ be such a graph
and let $G \to G'$ be a 1-extension move on the edge $e=(v_1v_2,b) $ for the colour case $b\to\{r,r\}$, where $G$ has no edge $(v_1v_2, r)$. Then the completely regular direction-length frameworks for $G'$ are minimally infinitesimally rigid.
\end{lem}

The following theorem can now be deduced from the construction scheme of Lemma \ref{l:construction2} and the discussion above.

\begin{thm}\label{t:direction_noneuclidean}
Let $G$ be a bi-coloured multigraph.
A completely regular non-Euclidean direction-length framework $(G,p)$ in $(\bR^2,\|\cdot\|_q)$, where $1<q<\infty, q\neq 2,$ is minimally rigid if and only if $G$ is $(2,2)$-tight and $(2,3)$-limited.
\end{thm}

\section{Closing remarks}\label{s:closing}

We note some further interesting double-distance contexts.
\medskip

\emph{1. Projection constraints.}  In Nixon, Owen and Power \cite{nix-owe-pow-2} $(2,1)$-tight combinatorial characterisations of  bar-joint frameworks in $\bR^3$ were obtained for the generic rigidity of frameworks whose joints were constrained to a cone or, more generally, to an algebraic surface of revolution which is not a cylinder. One can augment this setting to form the double-distance context of \emph{cone frameworks with projection constraints} as follows.
Let $\C$ be a cone in $\bR^3$ whose axis is orthogonal to the $xy$-plane. Let $d_r(p_1,p_2)$ denote the Euclidean distance between joints located on this cone and let $d_b(p_1, p_2)$ denote the distance between their projections on the $xy$-plane. The context $(\C, d_b, d_r)$ has a 1-dimensional space of rigid infinitesimal motions corresponding to rotation about the axis of the cone. In particular every completely regular framework with a complete (loopless) bi-coloured graph with $|V| \geq 3$ has a  $1$-dimensional space of infinitesimal flexes, and this space is given by infinitesimal rotations. Thus we have the following corollary of Theorem  \ref{t:topcount}.

\begin{cor}\label{c:Prigidityequivalence}
Let $(G,p)$ be a completely regular bar-joint framework for $(\C, d_b, d_r)$ with at least 3 joints. Then the following are equivalent:

(i) $(G,p)$ is continuously rigid;

(ii) $(G,p)$ is infinitesimally rigid;

(iii) $\rank Df_G(p) = 2|V|-1$.
\end{cor}

The underlying class of structure graphs for minimal infinitesimal rigidity are thus bi-coloured graphs which are $(2,1)$-tight and are $(2,3)$-\emph{limited} in the sense that any subgraph with only blue edges is $(2,3)$-sparse. We conjecture that a completely regular cone-and-projection framework $(G,p)$ is minimally infinitesimally rigid if and only if $G$ is $(2,1)$-tight and $(2,3)$-limited in this sense.

{In analogy with the cone-and-projection context, one may define projection constraints for any chosen surface. For example if one takes the surface to be the cylinder then the appropriate sparsity type is $(2,2)$-tight and $(2,3)$-limited. We gave a recursive construction of such graphs in Section \ref{s:EuclideanAndNon} and this provides  combinatorial methods that can be used to characterise minimal rigidity in the cylinder-and-projection context.  }

% { We also note the following double-distance contexts where the underlying combinatorics is as above and where we expect there to be a similar characterisation.
\medskip

\emph{2. Reflection distance in a half-plane. }
Let $(Y, d_b, d_r)$ be the double-distance context indicated in Section \ref{ss:furthercontexts} where $Y = [0,\infty)\times \bR$, %$Y_0 = (0,\infty)\times \bR$, 
with $d_b$ being Euclidean distance and $d_g$ the reflection distance where 
$d_g(p_1, p_2)$ is the minimum over boundary points $q$ of the sum 
$d_b(p_1, q)+d_b(p_2, q)$.
There is a $1$-dimensional space of rigid motion infinitesimal flexes, given by translational flexes and 
it follows from Theorem  \ref{t:topcount} that a minimally infinitesimally rigid completely regular framework for $(Y, d_b, d_r)$ has bi-coloured graph which is $(2,1)$-tight and $(2,3)$-limited for blue subgraphs, as before.

\medskip

\emph{3. The flat projective plane.}
Let $\P_2$ be the set $B_2/\sim$ where $B_2$ is the closed unit disc in $\bR^2$, centred at the origin, and $\sim$ is the equivalence relation for which two points $x, y$ of the closed disc are related
if and only if they are antipodal points on the boundary, so that $y=-x$. With the natural topology, inherited from the Euclidean topology, $\P_2$ is a compact Hausdorff topological space and is  homeomorphic to the real projective plane. {With this construction we may view the open disc $\bD$ as densely embedded in $\P_2$. Also the modulus $|p|$ of a point $p$ in $\P_2$ is well-defined in terms of this embedding and continuity. Thus a point has modulus $1$ if it corresponds to a boundary point of $\bD$.}

The following separation distances are well-defined on the product  $\bD \times \bD$ in $\P_2 \times \P_2$. \medskip

(i) The \emph{direct (Euclidean) distance}, $d_b(p,q)=\|p-q\|_2$.
\medskip

(ii) The \emph{re-entrant  distance}, $d_r(p,q)=\inf_{|x|=1}\{\|p-x\|_2+\|q+x\|_2\}$.
\medskip

(iii) The \emph{geodesic distance}, $d_g(p,q)=\mbox{ min} \{d_b(p,q), d_r(p,q)\}$.
\medskip

The re-entrant distance occurs as the sum of the lengths of two line segments to antipodal points on the boundary for one of the positions where the angles subtended by the corresponding diameter are equal.
%See Lemma \ref{l:extremal}  below.
While the direct and re-entrant distances do not have continuous extensions to
$\P_2 \times \P_2$ the geodesic distance does and $(\P_2, d_g)$ is a complete metric space. 

For the double-distance context  $(\P_2, d_b, d_r)$ there is  a 1-dimensional space of rigid infinitesimal motions, corresponding to rotation about the origin. Also it follows 
from Theorem  \ref{t:topcount}, once again, that a minimally infinitesimally rigid completely regular framework  has bi-coloured graph which is $(2,1)$-tight and $(2,3)$-limited and we conjecture that this graph condition is also sufficient.
%Further new methods are needed to determine 1-extension rigidity preservation in all colour cases.
\medskip

{ We remark that George and Ahmed \cite{Geo-Ahm} consider the rigidity of a different class of frameworks in a $1$-dimensional projective space. For their context the structure graphs are hypergraphs and the rigidity matrix is given in terms of cross ratios.
The conjecture for the combinatorial characterisation of the hypergraphs for minimal infinitesimal rigidity in this setting was proved by Jord\'{a}n and Kazsanitzky \cite{Jor-Kaz}.
}
\medskip

\emph{4. Elevation constraints.} %There are a number of further double-distance contexts that one may consider. For example 
Consider $\bR^d$ to be oriented with respect to a coordinate hyperplane
and normal direction. Then we may consider an associated double-distance context where $d_b$ is Euclidean distance and $d_r$ is an \emph{elevation separation} relative to the hyperplane. Thus for $\bR^2$ and the hyperplane $x=0$ we may define
\[
d_r(p_1,p_2) = (y_1-y_2)^2.
\]
We conjecture that the sparsity type characterising rigidity is that the bi-coloured structure graph should be $(2,2)$-tight containing a blue spanning tree with any blue subgraph being $(2,3)$-sparse and any red subgraph being $(1,1)$-sparse. \\

\emph{5. Global rigidity.} It would be interesting to consider global rigidity in double-distance contexts. On the cylinder, with Euclidean distances, global rigidity was characterised in \cite{JNglobal}. A key step was to understand `vertically restricted' frameworks on the cylinder. These are double-distance frameworks on the cylinder where the two types of distance constraints are Euclidean (blue) and `dilation' (red). More formally if $(G,p)$ and $(G,\hat p)$ are frameworks on the cylinder with $p(v_i)=(x_i,y_i,z_i)$ and $\hat p(v_i)=(\hat x_i,\hat y_i, \hat z_i)$, then the dilation constraint for an edge $v_iv_j$ says that for $(G,p)$ and $(G,\hat p)$ to be equivalent, we must have $z_i/z_j = \hat z_i / \hat z_j$. In particular \cite{JNglobal} characterised the bi-coloured graphs which are `vertically restricted' rigid on the cylinder when all possible vertical constraints are present (equivalently when the red subgraph is a spanning connected graph). In the general double-distance context one may conjecture that the sparsity type for minimal rigidity is that the bi-coloured structure graph should be $(2,1)$-tight with any blue subgraph being $(2,2)$-sparse and any red subgraph being $(1,1)$-sparse.\\

\emph{6. Coning.} Frameworks on the sphere are often understood by way of coning \cite{Whicones}. One can check that the following association provides double-distance cone bar-joint frameworks $(\tilde{G}, \tilde{p})$ which are equivalent to bar-joint frameworks in $\bR^3$ constrained to a cylinder.

Let $X = \bR^3$, let $d_1$ be the usual Euclidean distance
and let $d_2$ be the separation distance defined by 
\[
d_2((x_1, y_1, z_1),(x_2, y_2, z_2)) = ((x_1-x_2)^2
+(y_1-y_2)^2)^{\frac{1}{2}}.
\]
Let $G=(V,E)$ be a simple graph and define $\tilde{G} = (\tilde{V}, \tilde{E})$ where 
\[
\tilde{V} = V\cup\{v_0\}, \tilde{E} = E_1\cup E_2, E_1=E, 
E_2 = \{v_0\}\times V.  
\]
Let $(G,p)$ be a bar-joint framework constrained to the unit radius cylinder $\Y$ about the $z$-axis, and let $\tilde{p} = (p_0,p)$ with $p_0=(0,0,0)$. Then the double-distance bar-joint framework $(\tilde{G},\tilde{p})$ is naturally equivalent to $(G,p)$ on $\Y$. \\

\bibliographystyle{abbrv}
\def\lfhook#1{\setbox0=\hbox{#1}{\ooalign{\hidewidth
  \lower1.5ex\hbox{'}\hidewidth\crcr\unhbox0}}}

\end{document}